\UseRawInputEncoding
\documentclass{amsart}

\newtheorem{theorem}{Theorem}[section]
\newtheorem{lemma}[theorem]{Lemma}

\theoremstyle{definition}
\newtheorem{definition}[theorem]{Definition}

\newtheorem*{theorem A}{Theorem A}
\newtheorem*{theorem B}{Corollary B}
\newtheorem*{theorem C}{Corollary C}
\newtheorem*{theorem D}{Corollary D}
\newtheorem{proposition}{Proposition}

\theoremstyle{remark}

\numberwithin{equation}{section}

\begin{document}

\title{On the centralizers of rescaling separating differentiable vector fields}


\author{Bo Han}
\address{School of Mathematical Sciences, Beihang University, Beijing, 100191, P.R. China}
\email{hanbo@buaa.edu.cn}

\author{Xiao Wen}
\address{LMIB, Institute of Artificial Intelligence \& School of Mathematical Sciences, Beihang University, Beijing, 100191, P.R. China}
\address{ Beijing Zhongguancun Laboratory, Beijing, 100094, P.R.China}
\email{wenxiao@buaa.edu.cn}

\thanks{X.W. was partially supported by National Natural Science
Foundation of China (No. 12071018), National Key R\&D Program of China (No. 2022YFA1005801) and the Fundamental
Research Funds for the Central Universities.}


\subjclass[2020]{37D20, 37C10, 37C20}

\keywords{Centralizer, Expansiveness, Hyperbolicity}

\begin{abstract}
We introduce a new
version of expansiveness similar to separating property for flows. Let $M$ be a compact Riemannian manifold without boundary and $X$ be
a $C^1$ vector field on $M$ that generates a flow $\varphi_t$ on $M$.  We call $X$ {\it rescaling separating} on a compact invariant set $\Lambda$ of $X$ if there is $\delta>0$ such that, for any $x,y\in \Lambda$, if $d(\varphi_t(x), \varphi_{t}(y))\le \delta\|X(\varphi_t(x))\|$ for all $t\in \mathbb R$, then $y\in{\rm Orb}(x)$. We prove that if $X$ is rescaling separating on $\Lambda$ and every singularity of $X$ in $\Lambda$ is hyperbolic, then for any $C^1$ vector field $Y$, if the flow generated by $Y$ is commuting with $\varphi_t$ on $\Lambda$, then $Y$ is collinear to $X$ on $\Lambda$. As applications of the result, we show that the centralizer of a rescaling separating $C^1$ vector field without nonhyperbolic singularity is quasi-trivial and there is is an open and dense set $\mathcal{U}\subset\mathcal{X}^1(M)$ such that for any star vector field $X\in\mathcal{U}$, the centralizer of $X$ is collinear to $X$ on the chain recurrent set of $X$.
\end{abstract}

\maketitle

\section{Introduction}
In this paper, we study the centralizers of $C^1$ vector fields containing singularities. The study of centralizers of dynamical systems has appeared from 1970s, together with abundant results of hyperbolic dynamics. In 1970, Walters \cite{W} investigated the continuous transformations of metric spaces with discrete centralizers and unstable centralizers and proved that the expansive homeomorphisms have unstable centralizers. In 1973, Kato-Morimoto \cite{KM} proved that the centralizers of Anosov flows is pseudo-trivial. Then in 1976, Oka \cite{Oka} extended the above conclusions to the expansive flows, proving that the expansive flow has quasi-trivial centralizers open and densely. Next Palis-Yoccoz \cite{PY1} and Sad \cite{Sad} proved that the centralizers of $C^{\infty}$ systems with Axiom A plus strong transversal conditions have trivial centralizers on a $C^{\infty}$ open and dense subset.

In this note, we concentrate on the centralizes of flows with some weaker expansive properties. Following the idea of Oka, recently in 2018, Bonomo-Rocha-Varandas \cite{BRV} studied the centralizers of Komuro expansive flows and proved that the centralizers of $C^{\infty}$ Komuro expansive transitive flows plus singularities hyperbolic and non-resonance conditions are trivial. Martin Leguil, Davi Obata, and Bruno Santiago \cite{LOS} proved that if the singularities of kinematic expansive flow on a three-dimensional manifold were hyperbolic, the centralizer of the vector field trivial. Here we will give an extension of the above conclusions by weakening the definitions of expansiveness. In the weaker definitionss of expansiveness, we will consider the velocity of the flow near the singularity.

Let $M$ be a compact Riemannian manifold without boundary. Denote by $\mathcal{X}^r(M)$ the $C^r$ ($r\geq 1$) vector fields on $M$. Let $\varphi_t=\varphi_t^X$ be the flow generated by a vector field $X\in\mathcal{X}^1(M)$. Let $\Lambda$ be a compact invariant set of $\varphi_t$. Flow $\varphi_t|_{\Lambda}$ is {\it expansive} if for any $\epsilon>0$ there is $\delta>0$ such that, for any $x$ and $y$ in $\Lambda$ and any  continuous functions $\theta: \mathbb R\to \mathbb R$, if $d(\varphi_t(x), \varphi_{\theta(t)}(y))\le \delta$ for all $t\in \mathbb R$, then $\varphi_{\theta(t)}(y)\in \varphi_{[-\epsilon, \epsilon]}(\varphi_{t}(x))$ for some $t\in \mathbb R$. In \cite{Oka}, it is proved that if $\varphi_t|_{\Lambda}$ is expansive, then $\varphi_t|_{\Lambda}$ has a kind of trivial centralizer: if a flow $\psi_t$ on $\Lambda$ verifying $\psi_t\circ\varphi_t|_\Lambda=\varphi_t\circ\psi_t|_\Lambda$, then there is a continuous function $A:\Lambda\to\mathbb{R}$ such that $\psi_t(x)=\varphi_{A(x)t}(x)$ for any $x\in\Lambda$ and $t\in\mathbb{R}$. It is known that every fixed point of an expansive flow should be isolated(\cite{BW}), hence Oka's work does not work for some well known chaotic dynamical systems such as Cherry flows, Lorenz attractors, etc. In \cite{BRV} and \cite{BFH}, the condition of expansiveness in the conclusion of Oka was generalized by Komuro expansiveness and kinematic expansiveness plus the hyperbolicity of singularities. Here we say that $\varphi_t|_\Lambda$ is {\it Komuro expansive}, if for any $\varepsilon>0$, there is $\delta>0$ such that the following holds: for any $x,y\in\Lambda$ and any increasing homeomorphism $\theta:\mathbb{R}\to\mathbb{R}$, if $d(\varphi_t(x), \varphi_{\theta(t)}(y))<\delta$ for every $t\in\mathbb{R}$, then $\varphi_{\theta(t)}(y)\in\varphi_{[-\varepsilon,\varepsilon]}(\varphi_t(x))$ for some $t\in\mathbb{R}$. We say that $\varphi_t|_{\Lambda}$ is {\it kinematic expansive}, if for any $\varepsilon>0$, there is $\delta>0$ such that the following holds: for any $x,y\in\Lambda$, if $d(\varphi_t(x), \varphi_t(y))<\delta$ for every $t\in\mathbb{R}$, then $y\in\varphi_{[-\varepsilon,\varepsilon]}(x)$. In \cite{LOS}, it is shown that once we want to get the similar results for triviality of centralizers, we can weaken the expansiveness further by considering the following separating property. Flow $\varphi_t|_{\Lambda}$ is {\it separating} if there exists a constant $\delta>0$ such that for any $x,y\in \Lambda$, if $d(\varphi_t(x), \varphi_t(y))<\delta$ for all $t\in\mathbb{R}$, then $y\in{\rm Orb}(x)$. In this paper we will improve the conclusions of \cite{BRV, LOS, Oka} by considering rescaling separating property instead of separating property for $C^1$ vector fields.

\begin{definition}
Let $X\in\mathcal{X}^1(M)$ and $\varphi_t$ be the flow generated by $X$ and $\Lambda$ be a compact invariant set of $\varphi_t$. We say that $\varphi_t|_\Lambda$ is rescaled separating if there is $\varepsilon>0$ such that for any $x, y\in\Lambda$, if $d(\varphi_t(x), \varphi_t(y))\leq\varepsilon\|X(\varphi_t(x))\|$ for all $t\in\mathbb{R}$, then $y\in{\rm Orb}(x)$.
\end{definition}

Denote by
$$\mathcal{Z}^r(\varphi_t|_{\Lambda})=\{Y\in\mathcal{X}^r(M): \psi_t(\Lambda)=\Lambda,\ \ \ \psi_s\circ\varphi_t|_\Lambda=\varphi_t\circ\psi_s|_\Lambda, \text{ for all } t, s\in\mathbb{R},$$
$$\text{        \ \ \ \ \ \ \ \ \ \ \                where } \psi_t \text{ is the flow generated by } Y\}.$$
and $\mathcal{Z}^r(X)=\mathcal{Z}^r(\varphi_t|_M)$ for any positive integer $r\geq 1$. As usual, denote by ${\rm Sing}(X)=\{x\in M: X(x)=0\}$ be the singularities of $X$. In the paper we prove the following theorem firstly.

\begin{theorem A}\label{thmA}
Let $X\in\mathcal{X}^1(M)$ and $\varphi_t$ be the flow generated by $X$ and $\Lambda$ be a compact invariant set of $\varphi_t$ such that $\varphi_t|_{\Lambda}$ is rescaling separating. If every singularity of $X$ in $\Lambda$ is hyperbolic, then for any $Y\in  \mathcal{Z}^1(\varphi_t|_{\Lambda})$, there is a continuous map $A:\Lambda\setminus {\rm Sing}(X)\to\mathbb{R}$ which is constant along the orbit of $\varphi_t$ such that $$\psi_t(x)=\varphi_{A(x)t}(x)$$ for any $x\in\Lambda\setminus {\rm Sing}(X)$, where $\psi_t$ is the flow generated by $Y$.
\end{theorem A}

Recall that if for any $Y\in\mathcal{Z}^k(X)$, there exists a $C^1$ function $f:M\to\mathbb{R}$ with $X(f)\equiv 0$ such that $Y=f\cdot X$, then we say $X\in\mathcal{X}^r(M)$ has {\it quasi-trivial $C^k$-centralizer} ($1\leq k\leq r$). Here we can get the following corollary which improves Theorem A of \cite{LOS} by applying Theorem A.

\begin{theorem B}\label{thmB}
If $X\in\mathcal{X}^1(M)$ is rescaling separating and every singularity of $X$ is hyperbolic, then $X$ has quasi-trivial $C^k$-centralizer.
\end{theorem B}

In dimension three, under enough regularity assumptions, \cite{LOS} obtained triviality of centralizers, under the assumption of kinematic expansiveness. Here we can give a weaker version of kinematic expansiveness similar to rescaling separating property.

\begin{definition}
Let $\varphi_t$ be a flow generated by a vector fields $X\in\mathcal{X}^1(M)$ and $\Lambda$ be a compact invariant set of $\varphi_t$. If for any $\varepsilon>0$, there is $\delta>0$ such that for any $x\in\Lambda, y\in \Lambda$, if $d(\varphi_t(x), \varphi_t(y))\leq\delta\|X(\varphi_t(x))\|$ for every $t\in\mathbb{R}$, then $y\in\varphi_{[-\varepsilon,\varepsilon]}(x)$, then we say $\Lambda$ is rescaling kinematic expansive. If $M$ is rescaling kinematic expansive for $\varphi_t$, we say that $X$ or $\varphi_t$ is rescaling kinematic expansive.
\end{definition}

If for any $Y\in\mathcal{Z}^k(X)$, there exists constant $c\in\mathbb{R}$ such that $Y=cX$, then we say that $X$ has {\it trivial $C^k$-centralizer}. The following is a generalization of Theorem F in \cite{LOS}.

\begin{theorem C}\label{thmC}
Let $M$ be a $3$-dimensional compact Riemannian manifold without boundary, and $X\in\mathcal{X}^3(M)$. If $X$ is rescaling kinematic expansive and every singularity of $X$ is hyperbolic, then $X$ has trivial $C^3$-centralizer.

\end{theorem C}

The motivation of rescaling separating property or rescaling kinematic expansiveness is from the rescaling expansiveness proposed by \cite{WW}. It is well known that for differentiable dynamical systems, expansiveness is closely related to hyperbolicity. To unify the hyperbolicity of Anosov flow and Lorenz attractor, C. Morales, M. Pacifico and E. Pujals \cite{MPP} proposed the notion of singular hyperbolicity. Whereas Anosov flow and Lorenz attractor does not satisfy the standard expansiveness and the kinematic expansiveness, one know that every singular hyperbolic set is rescaling expansive (\cite{WW}) and then rescaling kinematic expansive. Hence our Theorem A can be applied for singular hyperbolic sets or star systems as follows.

Recall that a vector field $X\in\mathcal{X}^1(M)$ is called a star system or $X\in\mathcal{X}^{*1}(M)$ if $X$ has a neighborhood $\mathcal{U}$ in $\mathcal{X}^1(M)$ such that each $Y\in\mathcal{U}$ has only finitely many singularities and at most countably many periodic orbits (or equivalently, all singularities and periodic orbits of each $Y\in\mathcal{U}$ are hyperbolic \cite{L,M}). The flow generated by star vector field is called star flow.

Recall that a pair of sequences $\{x_i\in M:0\leq i\leq k\}$ and $\{t_i\in \mathbb{R}:0\leq i\leq k-1\}, \ k\geq1$, is an $\varepsilon$-{\it pseudo orbit} from $x_0$ to $x_k$ for a flow $\varphi_t$, if $$t_i\geq1 \ \hbox{and} \ d(x_{i+1}, \varphi_{t_i}(x_i))<\varepsilon,$$
for every $0\leq i\leq k-1$.
We say that $x\in M$ is {\it chain recurrent} if for every $\varepsilon>0$, there is an $\varepsilon$-pseudo orbit from $x$ to $x$. We call the set of chain recurrent points, the {\it chain recurrent set} and we denote it by $\mathcal{R}(X)$, which $X$ is a vector field which generates flow $\varphi_t$. Here we can also get the following corollary by applying Theorem A.

\begin{theorem D}\label{thmD}
There is a $C^1$ open and dense subset $\mathcal{U}$ of $\mathcal{X}^1(M)$ such that if $X\in\mathcal{U}$
is a star flow then for any $Y\in\mathcal{Z}^1(\varphi_t|_{\mathcal{R}(X)})$,  $Y$ is collinear to $X$ on the chain recurrent set $\mathcal{R}(X)$.
\end{theorem D}

\section{Proof of Theorem A}
As usual, denote by
$$T_xM(r)=\{v\in T_xM: \|v\|<r\},$$ $$B_r(x)=\exp_x(T_xM(r)).$$ By the compactness of $M$, we can fix a constant $a>0$ such that $$m(D_p\exp_x)>2/3, ~ \|D_p\exp_x\|<3/2$$ for any $p\in T_xM(a)$. By the $C^1$ smoothness of $X$,
there are constants $L>0$ and a $C^1$ neighborhood $\mathcal{U}_0$ of $X$ such that for any $Y\in\mathcal{U}_0$ and $x\in M$ the
vector fields
$$\bar{Y}=(\exp_x^{-1})_*(Y|_{B_{a}(x)})$$ in
$T_xM(a)$ are locally Lipschitz vector fields with a Lipschitz constant $L$. We call $L$ a {\it local Lipschitz constant} of $X$ with respect to the neighborhood $\mathcal{U}_0$.

\begin{lemma}\label{lem1}
Let $X\in\mathcal{X}^1(M)$ and $\mathcal{U}_0$ be given as above. Then for any $\delta>0$, there is $\mu>0$, such that for any $Y\in\mathcal{U}_0, |t|\leq\mu$ and any $x\in M$, one has
$$d(\psi_t(x), x)\leq\delta\|Y(x)\|,$$
where $\psi_t$ is the flow generated by $Y$.
\end{lemma}

\begin{proof}
Let $L$ be a local Lipschitz constant of $X$ with respect to $\mathcal{U}_0$. Let $Y$ be a vector field in $\mathcal{U}_0$ and $\psi_t$ be the flow generated by $Y$. It is well known that $\|Y(\psi_t(x))\|\leq e^{L|t|}\|Y(x)\|$ for any $t\in\mathbb{R}$. Given $\delta>0$, we can find $\mu>0$ satisfied $\mu e^{L\mu}<\delta$. For any $Y\in\mathcal{U}_0,|t|\le\mu$,
If $x$ is a singularity of $Y$, it is trivial that $d(\psi_t(x),x)\leq\delta\|Y(x)\|$. Now we assume that $x$ is not a singularity of $Y$ and $t\geq0$, one has
$$d(\psi_t(x),x)\leq\int^t_0\|\frac{d\psi_s(x)}{ds}\|ds=\int^t_0\|Y(\psi_s(x))\|ds$$
$$\leq\int^t_0e^{Ls}\|Y(x)\|ds\leq te^{Lt}\|Y(x)\|$$
$$\leq\mu e^{L\mu}\|Y(x)\|\leq\delta\|Y(x)\|.$$
Similar estimations holds for the case of $t\leq 0$. This ends the proof of the lemma.
\end{proof}

Recall that ${\rm Sing}(X)$ is the set of singularities of $X$. We call $x\in M$ a {\it regular point} if $x\in M\setminus {\rm Sing}(X)$. For a regular point $x\in M$ of $X$, denote  the {\it normal space} of $X(x)$ to be  $$N_x=N_x(X)=\{v\in T_x M:v\perp X(x)\}.$$ Given a constant $r>0$, we can take a box $$U_x(r\|X(x)\|)=\{v+tX(x)\in T_xM: v\in N_x, \|v\|\leq r\|X(x)\|, |t|\leq r\}$$ in $T_xM$. Define a $C^1$ map  $$F_x:U_x(r\|X(x)\|)\to M$$ to be $$F_x(v+tX(x))=\varphi_t(\exp_x(v)).$$
This map $F_x$ is called a {\it flowbox} of $X$ at $x$. In \cite{WW}, the following relative uniform version of flowbox theorem is proved.

\begin{proposition}[\cite{WW}, Proposition 2.2]\label{flowbox}
For any $C^1$ vector field $X$ on $M$, there is $0<r_0\leq \frac{1}{10L}$ such that for any regular point $x$ of ${X}$, $F_x:{U}_x(r_0\|X(x)\|)\to M$ is an embedding whose image contains no singularities of $X$, $m(D_pF_x)> 1/3$ and
$\|D_pF_x\|<3$ for every $p\in {U}_x(r_0\|X(x)\|)$.
\end{proposition}

From the above proposition we can easily see that for any $t_1, t_2\in[-r_0, r_0]$ and any $x\in M\setminus {\rm Sing}(X)$, we have $\varphi_{t_1}(x)\neq \varphi_{t_2}(x)$ when $t_1\neq t_2$.

\begin{lemma}\label{lem2}
Let $x\in M\setminus{\rm Sing}(X)$. Assuming that a continuous curve $\zeta:[0, 1]\to B_{\frac{r_0}{6}\|X(x)\|}(x)$ satisfies $\zeta(0)=x$ and $\zeta(t)\in {\rm Orb}(x)$ for all $t\in[0,1]$, then for any $t\in[0,1]$, there is $\eta\in [-\frac{r_0}{2}, \frac{r_0}{2}]$ such that $\zeta(t)=\varphi_{\eta}(x)$.
\end{lemma}

\begin{proof}
Denote by $\pi_x: T_xM\to N_x$ the orthogonal projection from $T_xM$ to $N_x$ for any $x\in M\setminus{\rm Sing}(X)$. Then we can define a differentiable map $P_x=\pi_x\circ F_x^{-1}:F_x({U}_x(r_0\|X(x)\|))\to N_x$. Note that we have $m(D_pF_x)\geq 1/3$ for any $p\in{U}_x(r_0\|X(x)\|)$, hence $B_{\frac{r_0}{6}\|X(x)\|}(x)\subset F_x({U}_x(r_0\|X(x)\|))\to N_x$. So we can get a continuous curve $P_x\circ\zeta$ in $N_x$. On the other hand, we know that
$$\{t\in\mathbb{R}: \varphi_t(x)\in B_{\frac{r_0}{6}\|X(x)\|}(x)\}$$
is an open set in $\mathbb{R}$ and then a union of countable open intervals $\{I_j\}$, and each $P_x(\varphi_{I_j}(x))$ is a single point in $N_x$. Hence we can see that the image of $P_x\circ\zeta$ is at mostly countable, hence the image of $P_x\circ\xi$ should be $\{0_x\}$. Then we know that $\zeta(t)\in P_x^{-1}(0_x)\subset\varphi_{[-\frac{r_0}{2}, \frac{r_0}{2}]}(x)$. This proves the lemma.
\end{proof}

\begin{lemma}\label{lem3}
Let $X\in\mathcal{X}^1(M)$ and $\varphi_t$ be the flow generated by $X$ and $\Lambda$ be a compact invariant set of $\varphi_t|_{\Lambda}$. For any $Y\in  \mathcal{Z}^1(\varphi_t|_{\Lambda})$ and any hyperbolic singularity $\sigma\in\Lambda$ of $X$, we have $\sigma\in {\rm Sing}(Y)$.
\end{lemma}
\begin{proof}
Let $Y\in  \mathcal{Z}^1(\varphi_t|_{\Lambda})$ and $\psi_t$ be the flow generated by $Y$. Since
$$\psi_s(\sigma)=\psi_s(\varphi_t(\sigma))=\varphi_t(\psi_s(\sigma))$$
for all $t\in\mathbb{R}$, hence $\psi_s(\sigma)$ is a fixed point of $\varphi_t$, then we know that $\psi_s(\sigma)=\sigma$ for all $s\in\mathbb{R}$ by the fact that $\sigma$ is an isolated fixed point of $\varphi_t$.
\end{proof}

\begin{lemma}\label{lem4}
Let $X\in\mathcal{X}^1(M)$ and $\varphi_t$ be the flow generated by $X$ and $\Lambda$ be a compact invariant set of $\varphi_t|_{\Lambda}$. Assume that every $\sigma\in{\rm Sing}(X)\cap \Lambda$ is hyperbolic, then there is a neighborhood $U$ of $\Lambda$ and a $C^1$ neighborhood $\mathcal{U}\subset\mathcal{X}^1(M)$ of $X$ and a constant $C>0$, such that for any $Y\in\mathcal{U}$ and any $z\in U$, one has
$$C^{-1}d(z, {\rm Sing}(Y|_U))\leq\|Y(z)\|\leq Cd(z, {\rm Sing}(Y|_U)).$$
\end{lemma}

\begin{proof}
Since $X$ is $C^1$ vector and the singularities of $X$ in $\Lambda$ are all hyperbolic, by the fact that every hyperbolic singularity is isolated, one can find a neighborhood $U$ of $\Lambda$ such that there exist only finite elements in ${\rm Sing}(X|_{\bar{U}})={\rm Sing}(X)\cap\Lambda$. Let ${\rm Sing}(X)\cap\Lambda=\{\sigma_1, \sigma_2, \cdots, \sigma_k\}$. Let $a>0$ be the constant such that $$m(D_p\exp_x)>2/3, ~ \|D_p\exp_x\|<3/2$$ for any $p\in T_xM(a)$ and $x\in M$. For any $\sigma_i (i=1,\cdots, k)$, denote by
$$\bar{X}_i=(\exp^{-1}_{\sigma_i})_*(X|_{B_a(\sigma_i)}).$$
Note that $\sigma_i$ is hyperbolic, hence $D_0X_i$ is invertible for every $i=1,\cdots, k$. Let
$$C_0=\max\{\max\{\|(D_0\bar{X}_i)^{-1}\|, \|D_0\bar{X}_i\|\}: i=1,\cdots, k\}.$$
We can take a neighborhood $U_i\subset B_a(\sigma_i)\cap U$ of $\sigma_i$ such that the vector filed $\bar{X}_i$ on $\exp^{-1}_{\sigma_i}(U_i)$ satisfies $$\|(D_v\bar{X}_i)^{-1}\|<2C_0, \ \ \ \|D_v\bar{X}_i\|<2C_0$$ for any $v\in\exp^{-1}_{\sigma_i}(U_i)$. Without loss of generality, we can also assume that the diameter of $U_i$ is less than the distance between $U_i, U_j$ for any $i\neq j\in\{1,\cdots,k\}$. Let $$m=\min\{\|X(x)\|:x\in \bar{U}\setminus(\bigcup\limits_{i=1}^kU_i)\},\ \  K=\max\{\|X(x)\|: x\in M\}$$ and $\rho=\min\{d(\sigma_i, M\setminus U_i): i=1, \cdots, k\}$. Then we can take a neighborhood $\mathcal{U}$ of $X$ such that for any $Y\in\mathcal{U}$, one has:
\begin{enumerate}
\item $\|(D_v\bar{Y}_i)^{-1}\|<4C_0,  \|D_v\bar{Y}_i\|<4C_0$ for any $v\in\exp^{-1}_{\sigma_i}(U_i)$, where $$\bar{Y}_i=(\exp^{-1}_{\sigma_i})_*(Y|_{B_a(\sigma_i)})$$ for $i=1,2,\cdots, k$;
\item ${\rm Sing}(Y|_U)=\{\sigma_i^Y: i=1,\cdots, k\}$, where $\sigma_i^Y\in U_i$ is the continuation of $\sigma_i$ with respect to $Y$;
\item $\min\{\|Y(x)\|:x\in \bar{U}\setminus(\bigcup\limits_{i=1}^kU_i)\}>\frac{m}{2}, \max\{\|Y(x)\|: x\in M\}<2K$ and $\min\{d(\sigma^Y_i, M\setminus U_i): i=1, \cdots, k\}>\frac{\rho}{2}$;
\end{enumerate}

Set $C=\max\{9C_0, \frac{4K}{\rho}, \frac{2{\rm diam}(M)}{m}\}$. Then for any $z\in U$, if $z\in U\setminus(\bigcup\limits_{i=1}^kU_i)$, then we have
$$\|Y(z)\|\leq 2K\leq C\cdot\frac{\rho}{2}<Cd(z, {\rm Sing}(Y|_U)),$$
$$\|Y(z)\|>\frac{m}{2}\geq C^{-1}{\rm diam}(M)\geq C^{-1}d(z, {\rm Sing}(Y|_U));$$
if $z\in U_i$ for some $i=1,\cdots, k$, on the one hand
$$\|Y(z)\|=\|(\exp_{\sigma_i})_*(\bar{Y}_i(\exp^{-1}_{\sigma_i}(z)))\|\leq \frac{3}{2}\|\bar{Y}_i(\exp^{-1}_{\sigma_i}(z))\|$$$$\leq \frac{3}{2}\cdot 4C_0\|\exp^{-1}_{\sigma_i}(z)-\exp^{-1}_{\sigma_i}(\sigma_i^Y)\|$$$$\leq\frac{3}{2}\cdot 4C_0\cdot \frac{3}{2}d(z, \sigma_i^Y)\leq Cd(z, {\rm Sing}(Y|_U)),$$
on the other hand
$$\|Y(z)\|=\|(\exp_{\sigma_i})_*(\bar{Y}_i(\exp^{-1}_{\sigma_i}(z)))\|\geq \frac{2}{3}\|\bar{Y}_i(\exp^{-1}_{\sigma_i}(z))\|$$$$\geq \frac{2}{3}\cdot (4C_0)^{-1}\|\exp^{-1}_{\sigma_i}(z)-\exp^{-1}_{\sigma_i}(\sigma_i^Y)\|$$$$\geq\frac{2}{3}\cdot (4C_0)^{-1}\cdot \frac{2}{3}d(z, \sigma_i^Y)\geq C^{-1}d(z, {\rm Sing}(Y|_U))$$
by the generalized mean value theorem. This ends the proof of the lemma.
\end{proof}

\begin{lemma}\label{lem5}
Let $\varphi_t$ be a flow generated by $X\in\mathcal{X}^1(M)$ and $\Lambda$ be a compact invariant set which is rescaling separating with respect to $\varphi_t$. If every singularity in $\Lambda$ is hyperbolic, then there is a $C^1$ neighborhood $\mathcal{U}$ of $X$ and constant $\mu>0$ such that for any flow $\psi_t$ generated by $Y\in \mathcal{Z}^1(\varphi_t|_{\Lambda})\cap\mathcal{U}$, there exist an unique function $z:[-\mu, \mu]\times(\Lambda\setminus{\rm Sing}(X))\to [-\frac{r_0}{2}, \frac{r_0}{2}]$ such that $\psi_s(x)=\varphi_{z(s, x)}(x)$ for any $(s, x)\in
[-\mu, \mu]\times(\Lambda\setminus{\rm Sing}(X))$. Moreover,
\begin{enumerate}
\item $z$ is continuous;
\item $z(t+s, x)=z(t,x)+z(s, \psi_t(x))$ for any $x\in\Lambda\setminus {\rm Sing}(X)$ and $t,s\in[-\mu,\mu]$ with $t+s\in[-\mu,\mu]$;
\item $z(s, \varphi_t(x))=z(s, x)$ for any $x\in\Lambda\setminus {\rm Sing}(X)$ and $s\in[-\mu,\mu]$ and $t\in\mathbb{R}$;
\item $z(s, x)=A(x)s$ for any $x\in\Lambda\setminus{\rm Sing}(X)$ and $s\in[-\mu, \mu]$, where $$A(x)=\mu^{-1}z(\mu, x).$$
\end{enumerate}
\end{lemma}

\begin{proof}
By Lemma \ref{lem4}, we can fix a neighborhood $\mathcal{U}\subset \mathcal{U}_0$ of $X$ and neighborhood $U$ of $\Lambda $ with constant $C$ such that for any $Y\in\mathcal{U}$, one has
$$C^{-1}d(z, {\rm Sing}(Y|_U))\leq\|Y(z)\|\leq Cd(z, {\rm Sing}(Y|_U)),\ \ \ \ \forall z\in U.$$
By shrinking $\mathcal{U}$ and $U$ if necessary, since every singularity $\sigma\in{\rm Sing}(X)\cap\Lambda$ is hyperbolic, we can assume that the cardinal number of ${\rm Sing}(Y|_U)$ is equal to the cardinal number of  ${\rm Sing}(X|_{U})={\rm Sing}(X)\cap\Lambda$. Then by Lemma \ref{lem3} we know that ${\rm Sing}(Y|_U)={\rm Sing}(X|_U)$ for any $Y\in \mathcal{Z}^r(\varphi_t|_{\Lambda})\cap\mathcal{U}$. Thus we have
$$\|Y(z)\|\leq C d(z, {\rm Sing}(Y|_U))=C d(z, {\rm Sing}(X|_U))\leq C^2\|X(z)\|$$
for any $z\in U$. Note here we are considering that $\Lambda\cap{\rm Sing}(X)$ is not empty, if $\Lambda\cap{\rm Sing}(X)=\emptyset$, we can also easily find $C, U$ and $\mathcal{U}$ such that $$\|Y(z)\|\leq C^2\|X(z)\|$$
holds for any $z\in U$ and $Y\in\mathcal{U}$.

Let $\delta>0$ be the constants in the definition of rescaling separating property. Without loss of generality we can assume that $\delta<\frac{r_0}{6}$. By Lemma \ref{lem1}, there is $\mu>0$ such that for any flow $\psi_t$ generated by a vector field $Y\in\mathcal{U}$ and any $|t|\leq \mu$, one has $d(\psi_t(x), x)\leq C^{-2}\delta\|Y(x)\|$ for any $x\in M$.

Now let $\psi_t$ be a flow generated by a vector field $Y\in \mathcal{Z}^1(\varphi_t|_{\Lambda})\cap\mathcal{U}$. By the fact $\psi_s\circ\varphi_t|_{\Lambda}=\varphi_t\circ\psi_s|_{\Lambda}$ for any $t,s\in\mathbb{R}$, we know that
$$d(\varphi_t(x), \varphi_t(\psi_s(x)))=d(\varphi_t(x), \psi_s(\varphi_t(x)))\leq C^{-2}\delta\|Y(\varphi_t(x))\|\leq \delta\|X(\varphi_t(x))\|$$
for all $x\in\Lambda$ and $t\in\mathbb{R}$ and $|s|\leq \mu$.

Fix $x\in\Lambda\setminus {\rm Sing}(X)$. By the rescaling separating property of $\Lambda$, we know that $\psi_s(x)\in{\rm Orb}(x)$ for all $s\in[-\mu,\mu]$. Note that $\psi_s(x)\in B_{\delta\|X(x)\|}(x)\subset B_{\frac{r_0}{6}\|X(x)\|}(x)$ for all $s\in[-\mu, \mu]$ and $\psi_0(x)=x$, by Lemma \ref{lem2} we know that there is $\eta=z(s, x)\in[-\frac{r_0}{2}, \frac{r_0}{2}]$ such that
$$\psi_s(x)=\varphi_{\eta}(x)$$ for all $s\in[-\mu,\mu]$. By Proposition \ref{flowbox} we know that for any $t_1, t_2\in[-\frac{r_0}{2}, \frac{r_0}{2}]\subset[-r_0, r_0]$, when $t_1\neq t_2$ we have $\varphi_{t_1}(x)\neq\varphi_{t_2}(x)$ for any $x\in M\setminus {\rm Sing}(X)$. Hence $\eta=z(s, x)$ is uniquely defined on $(s, x)\in[-\mu,\mu]\times(\Lambda\setminus {\rm Sing}(X))$. This gives the function $$z:[-\mu,\mu]\times(\Lambda\setminus {\rm Sing}(X))\to[-\frac{r_0}{2}, \frac{r_0}{2}].$$
$$(s,x)\mapsto\eta=z(s,x)$$
If $\eta=z(s,x)$ is not continuous, then one can find a sequence of $\{(s_n, x_n)\}$ in $[-\mu,\mu]\times(\Lambda\setminus {\rm Sing}(X))$ with $(s_n, x_n)\to (s_0, x_0)\in[-\mu,\mu]\times(\Lambda\setminus {\rm Sing}(X))$ as $n\to\infty$ such that $$|z(s_n, x_n)-z(s_0, x_0)|\not\rightarrow 0.$$ By choosing a subsequence we can assume that $z(s_n, x_n)-z(s_0, x_0)\to \eta_0\in[-r_0, r_0]$. Since $\eta_0\neq 0$, we have
$$\lim\limits_{n\to\infty}d(\varphi_{z(s_n, x_n)}(x_0), \varphi_{z(s_0, x_0)}(x_0))=d(\varphi_{z(s_0, x_0)+\eta_0}(x_0), \varphi_{z(s_0, x_0)}(x_0))\neq 0.$$
On the other hand we have
$$d(\varphi_{z(s_n, x_n)}(x_0), \varphi_{z(s_0, x_0)}(x_0))$$$$\leq d(\varphi_{z(s_n, x_n)}(x_n), \varphi_{z(s_n, x_n)}(x_0))+d(\varphi_{z(s_n, x_n)}(x_n), \varphi_{z(s_0, x_0)}(x_0))$$
$$=d(\varphi_{z(s_n, x_n)}(x_n), \varphi_{z(s_n, x_n)}(x_0))+d(\psi_{s_n}(x_n), \psi_{s_0}(x_0))$$
Since $|z(s_n, x_n)|$ is bounded and $d(x_n, x_0)\to 0$ we know $$d(\varphi_{z(s_n, x_n)}(x_n), \varphi_{z(s_n, x_n)}(x_0))\to0.$$ Since $s_n\to s_0$ and $x_n\to x_0$ we have $$d(\psi_{s_n}(x_n), \psi_{s_0}(x_0))\to 0,$$ thus we have
$$\lim\limits_{n\to\infty}d(\varphi_{z(s_n, x_n)}(x_0), \varphi_{z(s_0, x_0)}(x_0))=0,$$ a contradiction. This proves that $\eta=z(s,x)$ is continuous on $[-\mu,\mu]\times(\Lambda\setminus {\rm Sing}(X))$.

Note that
$$\varphi_{z(t+s,x)}(x)=\psi_{t+s}(x)=\psi_s(\psi_t(x))$$$$=\varphi_{z(s, \psi_t(x))}(\psi_t(x))=\varphi_{z(s, \psi_t(x))}(\varphi_{z(t,x)}(x))$$$$=\varphi_{z(t,x)+z(s,\psi_t(x))}(x),$$
and $z(t,x)+z(s,\psi_t(x))\in[-r_0, r_0]$, then we can see that $$z(t+s,x)=z(t,x)+z(s,\psi_t(x))$$ for any $x\in\Lambda\setminus{\rm Sing}(X)$ and $t,s\in[-\mu,\mu]$ with $t+s\in[-\mu,\mu]$. This proves item (2) of the lemma.

Fix $s\in[-\mu,\mu]$ and $t\in[-\frac{r_0}{2}, \frac{r_0}{2}]$ and $x\in\Lambda\setminus {\rm Sing}(X)$. Note that
$$\varphi_{t+z(s, \varphi_t(x))}(x)=\varphi_{z(s, \varphi_t(x))}(\varphi_t(x))$$$$=\psi_s(\varphi_t(x))=\varphi_t(\psi_s(x))$$$$=\varphi_t(\varphi_{z(s,x)}(x))=\varphi_{t+z(s, x)}(x).$$
By the fact that $|t+z(s, \varphi_t(x))|\leq r_0$ and $|t+z(s, x)|\leq r_0$ we know that $$t+z(s, \varphi_t(x))=t+z(s, x)$$ and then $z(s, \varphi_t(x))=z(s,x)$. Let $s\in[-\mu,\mu]$ and $t\in\mathbb{R}$ and $x\in\Lambda\setminus {\rm Sing}(X)$ be given. We can find $n\in\mathbb{N}$ big enough such that $|n^{-1}t|\leq\frac{r_0}{2}$, then we can see that
$$z(s, \varphi_t(x))=z(s, \varphi_{\frac{n-1}{n}t}(x))=\cdots=z(s, \varphi_{\frac{1}{n}t}(x))=z(s,x).$$
This proves item (3) of the lemma.

From item (2) and (3) we can see that for any $s, t\in[-\mu,\mu]$ and any $x\in\Lambda\setminus{\rm Sing}(X)$, we have
$$z(s+t, x)=z(t, x)+z(s, \psi_t(x))=z(t, x)+z(s, \varphi_{z(t, x)}(x))=z(t, x)+z(s, x).$$
Fix $x\in\Lambda\setminus{\rm Sing}(X)$. For any $n\in\mathbb{Z}^+$, we have $nz(n^{-1}\mu, x)=z(\mu, x)$, and then we have $$z(n^{-1}\mu, x)=n^{-1}z(\mu, x)=\mu A(x).$$ And then we have $$z(\frac{m}{n}\mu, x)=\frac{m}{n}\mu A(x)$$ for any rational number $\frac{m}{n}\in[0, 1]$. Note that $z(-t, x)=-z(t, x)$, we can see that $$z(\frac{m}{n}\mu, x)=\frac{m}{n}\mu A(x)$$ for any rational number $\frac{m}{n}\in[-1,1]$, by the continuity of $z(s, x)$ we can see that $$z(s, x)=A(x)s$$ for any $s\in[-\mu, \mu]$. This proves item (4).
\end{proof}

\begin{proposition}\label{pro2}
Let $\varphi_t$ be a flow generated by $X\in\mathcal{X}^1(M)$ and $\Lambda$ be a compact invariant set which is rescaling separating with respect to $X$. If every singularity in $\Lambda$ is hyperbolic, then there is a $C^1$ neighborhood $\mathcal{U}$ of $X$ such that for any flow $\psi_t$ generated by $Y\in \mathcal{Z}^1(\varphi_t|_{\Lambda})\cap\mathcal{U}$, there is a continuous function $A: \Lambda\setminus{\rm Sing}(X)\to \mathbb{R}$ which is constant along the orbit of $\varphi_t$ such that $$\psi_t(x)=\varphi_{A(x)t}(x)$$ for any $x\in\Lambda\setminus {\rm Sing}(X)$ and $t\in\mathbb{R}$.
\end{proposition}

\begin{proof}
Let $\mathcal{U}$ be given as in Lemma \ref{lem5}. Fix the flow $\psi_t$ generated by a vector field $Y\in\mathcal{Z}^1(\varphi_t|_{\Lambda})\cap\mathcal{U}$. Then we can take $z(s, x)=A(x)s$ as in Lemma \ref{lem5} for any $(s, x)\in [-\mu,\mu]\times(\Lambda\setminus{\rm Sing}(X))$.
By the continuity of $z(s, x)$ we can see that $A(x)$ is continuous on $\Lambda\setminus{\rm Sing}(X)$. Let $x\in\Lambda\setminus {\rm Sing}(X)$ and $t\in\mathbb{R}$ be given, we have
$$A(\varphi_t(x))=\mu^{-1}z(\mu, \varphi_t(x))=\mu^{-1}z(\mu, x)=A(x).$$
Hence $A(x)$ is constant along orbit of $\varphi_t$.

By the fact that $\psi_t(x)=\varphi_{z(t,x)}(x)$ we can easily see that $\psi_t(x)=\varphi_{A(x)t}(x)$ is true for any $x\in\Lambda\setminus{\rm Sing}(X)$ and $t\in[-\mu, \mu]$. Fix any $x\in\Lambda\setminus{\rm Sing}(X)$ and $t\in\mathbb{R}$, we can take $n\in\mathbb{N}$ big enough such that $|n^{-1}t|\leq \mu$. Denote by $\tau=n^{-1}t$, then we have
$$\psi_{2\tau}(x)=\psi_\tau(\psi_\tau(x))=\varphi_{A(\psi_{\tau}(x))\tau}(\varphi_{A(x)\tau}(x))$$$$=\varphi_{A(\varphi_{A(x)\tau}(x))\tau}(\varphi_{A(x)\tau}(x))=\varphi_{2A(x)\tau}(x),$$
$$\psi_{3\tau}(x)=\psi_\tau(\psi_{2\tau}(x))=\varphi_{A(\psi_{2\tau}(x))\tau}(\varphi_{2A(x)\tau}(x))$$$$=\varphi_{A(\varphi_{2A(x)\tau}(x))\tau}(\varphi_{2A(x)\tau}(x))=\varphi_{3A(x)\tau}(x),$$
and then by induction we have $$\psi_t(x)=\psi_{n\tau}(x)=\varphi_{nA(x)\tau}(x)=\varphi_{A(x)t}(x).$$ This ends the proof of Proposition \ref{pro2}.
\end{proof}

\begin{lemma}\label{lem6}
Let $\varphi_t$ be a flow generated by $X\in\mathcal{X}^r(M)$ and $\Lambda$ be a compact invariant set of $\varphi_t$. For any $Y\in \mathcal{Z}^k(\varphi_t|_{\Lambda})$ and $c\in\mathbb{R} (1\leq k\leq r)$,  we have $X+cY\in\mathcal{Z}^k(\varphi_t|_{\Lambda})$.
\end{lemma}

\begin{proof}
Let $\psi_t$ be the flow generated by $Y$. Let us consider a flow $\tilde{\psi}_t|_\Lambda=\varphi_t\circ\psi_{ct}|_\Lambda$. Since $\varphi_t, \psi_s$ is commuting on $\Lambda$, we have
$$\tilde\psi_{t+s}(x)=\varphi_{t+s}\circ\psi_{ct+cs}(x)=\varphi_t\circ\varphi_s\circ\psi_{ct}\circ\psi_{cs}(x)=\varphi_t\circ\psi_{ct}\circ\varphi_s\circ\psi_{cs}(x)=\tilde\psi_{t}\circ\tilde\psi_s(x),$$
for any $t, s\in\mathbb{R}, x\in\Lambda$, thus $\tilde\psi_t|_{\Lambda}$ is really a flow on $\Lambda$. For any $x\in\Lambda$ we have
$$\frac{{\rm d}}{{\rm d}t}\tilde\psi_t(x)|_{t=0}=\frac{{\rm d}}{{\rm d}t}\varphi_t(\psi_{ct}(x))|_{t=0}=X(x)+cY(x),$$
hence $\tilde\psi_t|_\Lambda$ is the restriction of the flow generated by $X+cY$ on $\Lambda$. For any $s, t\in\mathbb{R}$, we have
$$\tilde\psi_s\circ\varphi_t|_{\Lambda}=\varphi_s\circ\psi_{cs}\circ\varphi_t|_{\Lambda}=\varphi_s\circ\varphi_t\circ\psi_{cs}|_{\Lambda}=\varphi_t\circ\varphi_s\circ\psi_{cs}|_{\Lambda}=\varphi_t\circ\tilde\psi_s|_{\Lambda},$$
hence $X+cY\in\mathcal{Z}^k(\varphi_t|_{\Lambda})$.
\end{proof}

\noindent{\it Proof of Theorem A.} Let $\varphi_t$ be a flow generated by $X\in\mathcal{X}^1(M)$ and $\Lambda$ be a compact invariant set which is rescaling separating with respect to $\varphi_t$. Assume that every $\sigma\in\Lambda\cap{\rm Sing}(X)$ is hyperbolic. Let $Y\in \mathcal{Z}^1(\varphi_t|_{\Lambda})$. Take the neighborhood $\mathcal{U}$ of $X$ being given as in Proposition \ref{pro2}. Then we can take $c>0$ small enough such that $Y'=X+cY\in\mathcal{U}$. Let $\psi'_t$ be the flow generated by $Y'$. By Proposition \ref{pro2} we know that there exist a continuous function $A':\Lambda\setminus{\rm Sing}(X)\to\mathbb{R}$ such that $\psi'_t(x)=\varphi_{A'(x)t}(x)$ for any $x\in\Lambda\setminus{\rm Sing}(X)$ and $t\in\mathbb{R}$. Then we have
$$Y'(x)=\frac{{\rm d}}{{\rm d}t}\psi'_t(x)|_{t=0}=\frac{{\rm d}}{{\rm d}t}\varphi_{A'(x)t}(x)|_{t=0}=A'(x)X(x)$$
for any $x\in\Lambda\setminus{\rm Sing}(X)$. Then we have $$Y(x)=c^{-1}(Y'(x)-X(x))=c^{-1}(A'(x)-1)X(x)$$ for any $x\in\Lambda\setminus{\rm Sing}(X)$. Take $A(x)=c^{-1}(A'(x)-1)$, we can see that $A(x)$ satisfies the requests of Theorem A. This ends the proof of Theorem A.

\section{Proof of Corollaries}
Fix $1\leq k\leq r$. We say that $X\in\mathcal{X}^r(M)$ has {\rm collinear $C^k$-centralizers} if for any $Y\in\mathcal{Z}^k(X)$, one has
$${\rm dim} \langle X(x), Y(x)\rangle \leq 1$$
for any $x\in M$. The following lemma is Theorem 3.4 of \cite{LOS}.

\begin{lemma}\label{lem8}\cite{LOS}
Let $X\in\mathcal{X}^1(M)$. If $X$ has collinear centralizer and all the singularities of $X$ are hyperbolic, then $X$ has quasi-trivial $C^1$-centralizer.
\end{lemma}

Now we give the proof of Corollary B.

\bigskip

\noindent{\it Proof of Corollary B.} Since $X\in\mathcal{X}^1(M)$ is rescaling separating and every singularity of $X$ is hyperbolic,
we know that $X$ has collinear centralizer by Theorem A and then by Lemma \ref{lem8}, $X$ has quasi-trivial $C^1$-centralizer. This ends the proof of Corollary B.

\bigskip

The following is Proposition 4.9 of \cite{LOS}.

\begin{lemma}\label{lem10}
Let $\mathbb{T}^2$ denote the two dimensional torus. If $X\in\mathcal{X}^2(\mathbb{T}^2)$ and ${\rm Sing}(X)=\emptyset$,
then $X$ is not kinematic expansive.
\end{lemma}

Now we give the proof of Corollary C.

\bigskip

\noindent{\it Proof of Corollary C.} It is easy to see that rescaling kinematic expansiveness implies rescaling separating property. Hence if $X$ is rescaling kinematic expansive and all the singularities of $X$ are hyperbolic, by Corollary B, we know that $X$  has quasi-trivial centralizer. For any $Y\in\mathcal{Z}^3(X)$, we can take a function $f:M\to\mathbb{R}$ be a $C^1$, $X(f)\equiv0$ such that $f|_{M\setminus{\rm Sing}(X)}$ is $C^3$. Next, we will prove that $f$ is constant. If $f$ were not constant, there exist two real numbers $a<b$ such that $f(M)=[a, b]$. Note that all the singularities of $X$ are hyperbolic, hence there are at most finitely many of them. In particular, there exists a
non-trivial open interval $I\subset\mathbb{R}$ such that $I\subset f(M)\setminus f({\rm Sing}(X))$. Note that $f:M\setminus {\rm Sing}(X)\to\mathbb{R}$ is $C^3$, by Morse-Sard theorem, almost every value in $I$ is a regular value.
Take a regular value $c\in I$. Hence, $S_c:=f^{-1}({c})$ is a compact surface that does not
contain any singularity of $X$. Furthermore, since $f$ satisfied that $X(f)\equiv0$, then invariant along orbit of $X$, we know that $S_c$ is invariant set of $\varphi_t$, hence $X|_{S_c}$ is a $C^3$
non-singular vector field on $S_c$. Since $S_c$ is a two dimension manifold without boundary, and $X|_{S_c}$ is non-singular, by Poincar${\rm \acute{e}}$-Hopf theorem we know that $S_c$ is the torus $\mathbb{T}^2$ or the Klein bottle $2\mathbb{P}^2$. Note that $X$ is rescaling kinematic expansive on $S_c$ and $\|X(x)\|$ has a positive lower bound on $S_c$, we can see that the $\varphi_t|_{S_c}$ is kinematic expansive. Denote by $\tilde{\varphi}_t$ the flow on $\mathbb{T}^2$ which is the lift of $\varphi_t|_{S_c}$. Then we can easily check that $\tilde{\varphi}_t$ is also kinematic expansive, this contradict with Lemma \ref{lem10}. We conclude that $f$ is constant, and this implies that the $C^3$-centralizer of $X$ is trivial. This ends the proof of Corollary C.

\bigskip

\noindent{\it Proof of Corollary D.} According to Theorem 4 of \cite{BL}, there is an open dense set $\mathcal{U}\subset\mathcal{X}^1(M)$ such that the chain recurrent set $\mathcal{R}(X)$ of star flow $X\in\mathcal{U}$ is multisingular hyperbolic, hence the chain recurrent set $\mathcal{R}(X)$ is rescaling expansive by Theorem A of \cite{WW}, that is, there is $\delta_0>0$ such that for any $\delta\in(0,\delta_0]$, if $x,y\in\mathcal{R}(X)$ and an increasing continuous function $\theta:\mathbb{R}\to\mathbb{R}$ satisfy $d(\varphi_t(x), \varphi_{\theta(t)}(y))\leq \delta\|X(\varphi_t(x))\|$ for all $t\in\mathbb{R}$, then $y\in\varphi_{[-3\delta, 3\delta](x)}$, hence the chain recurrent set $\mathcal{R}(X)$ is rescaling kinematic expansive, and then rescaling separating. By applying Theorem A we know that for any $Y\in\mathcal{Z}^1(\varphi_t|_{\mathcal{R}(X)})$, there is a continuous function $A:\mathcal{R}(X)\setminus {\rm Sing}(X)\to \mathbb{R}$ such that $Y(x)=A(x)X(x)$ for any $x\in\mathcal{R}(X)$. Hence for any $x\in \mathcal{R}(X)$ we have $Y(x)$ is collinear to $X(x)$. This ends the proof of Corollary D.

\bibliographystyle{amsplain}

\begin{thebibliography}{10}


\bibitem {BFH} L. Bakker, T. Fisher, B. Hasselblatt, \textit{Centralizers of hyperbolic and kinematic-expansive flows.} Math. Res. Rep. \textbf{2} (2021), 21--44.

\bibitem {BF} L. Bakker, T. Fisher, \textit{Open sets of diffeomorphisms with trivial centralizer in the $C^1$-topology.}  Nonlinearity. \textbf{27} (2014), 2869--2885.

\bibitem {BW} R. Bowen, P. Walters, \textit{Expansive one-parameter flows.} J. DifferentialEquations. \textbf{12} (1972), 180--193.

\bibitem {BL} C. Bonatti, A. D. Luz, \textit{Star fows and multisingular hyperbolicity.} Journal of the European Mathematical Society. \textbf{23(8)} (2017), 2649--2705.

\bibitem {BRV} W. Bonomo, J. Rocha, P. Varandas, \textit{The centralizer of Komuro-expansive flows and expansive $\mathbb{R}^d$ actions.} Math. Z. \textbf{289} (2018), 1059--1088.

\bibitem {BCW} C. Bonatti, S. Crovisier, A. Wilkinson, \textit{The $C^1$ generic diffeomorphism has trivial centralizer.} publications Math¨¦matiques de l¡¯IH¨¦S. \textbf{109(1)} (2009), 185--244.

\bibitem {F1} T. Fisher, \textit{Trivial centralizers for axiom A diffeomorphisms.} Nonlinearity. \textbf{21(11)} (2008), 2505--2517.

\bibitem {LOS} M. Leguil1, D. Obata, B. Santiago, \textit{On the centralizer of vector fields: criteria of triviality and genericity results.} Math. Z. \textbf{297} (2021), 283--337.

\bibitem {K} N. Kopell, \textit{Commuting diffeomorphisms, Globa Analysis} Proc. Sympos. Pure Math. \textbf{XIV} (1970), 165--184.

\bibitem {KM} K. Kato, A. Morimoto, \textit{Topological stability of Anosov flows and their centralizers.} Topology. \textbf{12} (1973), 255--273.

\bibitem {L} S. T. Liao, \textit{The Qualitative Theory of Differential Dynamical Systems.} Science Press, 1996.

\bibitem {M} L. Markus, \textit{Structurally stable differential systems.} Ann. of Math. \textbf{73} (1961), 1--19.

\bibitem {MPP} C. Morales, M. Pacifico, E. Pujals, \textit{Robust transitive singular sets for 3-flows are partially hyperbolic attractors or repellers.} Ann. of Math. \textbf{160} (2004), 375--432.

\bibitem {Oka} M. Oka, \textit{Expansive flows and their centralizers.} Nagoya Math. J. \textbf{64} (1976), 1--15.

\bibitem {PY1} J. Palis, J. C. Yoccoz, \textit{Centralizers of Anosov diffeomorphisms on tori.} Ann. Sci. ¨¦cole Norm. Sup. \textbf{(4)22(1)} (1989), 98--108.

\bibitem {P} R. Plykin, \textit{On the structure of centralizers of Anosov diffeomorphisms of a torus.} Uspekhi Mat. Nauk. \textbf{53(6)} (1998), 259-260.

\bibitem {Sad} P. Sad, \textit{Centralizers of vector fields.} Topology. \textbf{18} (1979), 97--104.

\bibitem {SS1} S. Smale, \textit{Dynamics retrospective: great problems, attempts that failed.} Nonlinear science: the next decade, Los Alamos, NM, 1990. Phys. D 51. \textbf{52} (1991), 267--273.

\bibitem {SS2} S. Smale, \textit{Mathematical problems for the next century.} Math. Intell. \textbf{20} (1998), 7--15.

\bibitem {W} P. Walters, \textit{Homeomorphisms with Discrete Centralizers and Ergodic Properties.} Math. System Theory. \textbf{4(4)} (1970), 322--326.

\bibitem {WW} X. Wen, L. Wen, \textit{A rescaled expansiveness of flows.} Trans. Amer. Math. Soc. \textbf{371(5)} (2019), 3179--3207.






\end{thebibliography}

\end{document}